\documentclass[11pt,reqno,notitlepage]{amsart}

\usepackage{booktabs}                                 
\usepackage[english]{babel}
\usepackage{latexsym}
\usepackage[utf8]{inputenc}
\usepackage{babel}
\usepackage{csquotes}
\usepackage{fancyhdr}
\usepackage{longtable}
\usepackage{mathrsfs}
\usepackage{geometry}
\usepackage{comment}
\usepackage{textcomp}
\usepackage{caption}
\usepackage{lmodern}
\usepackage{mathrsfs}
 \usepackage[T1]{fontenc}

\usepackage[all,cmtip]{xy}
\usepackage{epsfig}

\usepackage{amsmath,amsfonts,amssymb,amsthm}
\usepackage{graphics}
\usepackage{graphicx}
\usepackage{verbatim}
\usepackage{dsfont}
\usepackage{enumerate}  
\usepackage{pdflscape}
\usepackage{longtable}
\usepackage{booktabs}
\usepackage{colortbl}

\usepackage{tikz-cd}
\usepackage{ytableau}

\usepackage[backref=page,style=alphabetic,sorting=nyt]{biblatex}
\usepackage[shortlabels]{enumitem}
\usepackage[breaklinks]{hyperref} 
\hypersetup{
	colorlinks,
	linkcolor={black},
	citecolor={blue},
	urlcolor={blue},
}
\usepackage{stmaryrd}
\usepackage{multirow}
\usepackage{longtable}

\usepackage{xcolor,colortbl}
\definecolor{lavender}{rgb}{0.9, 0.9, 0.98}

\usepackage{enumitem}   
\usepackage{hyperref}   

\makeatletter
\newcommand{\itemtag}[1]{%
  \makebox[1.6cm][c]{(#1)} & \phantomsection\edef\@currentlabel{(#1)}\label{item_#1}%
}
\makeatother
\makeatletter
\renewcommand{\itemtag}[1]{%
  (#1) & \phantomsection\edef\@currentlabel{(#1)}\label{item_#1}%
}
\makeatother

\theoremstyle{plain}
\newtheorem{thm}{Theorem}[section]

\newtheorem{proposition}[thm]{Proposition}
\newtheorem{lemma}[thm]{Lemma}
\theoremstyle{definition}

\newtheorem{rmk}[thm]{Remark}

\newcommand{\bigslant}[2]{{\raisebox{.2em}{$#1$}\left/\raisebox{-.2em}{$#2$}\right.}}

\newcommand{\OGr}{\mathrm{OGr}}
\newcommand{\SGr}{\mathrm{SGr}}
\newcommand{\Gr}{\mathrm{Gr}}
\newcommand{\cU}{\mathcal{U}}
\newcommand{\kuchle}{K\"uchle}
\newcommand{\cS}{\mathcal{S}}

\usepackage{comment}
\usepackage{todonotes}

\usepackage[capitalise,nameinlink]{cleveref}
\crefname{table}{Table}{Tables}
\crefname{section}{Section}{Sections}
\crefname{mystyle}{Fano}{Fanos}
\crefname{claim}{Claim}{Claims}
\crefname{rmk}{Remark}{Remarks}
\crefname{workhyp}{WH}{WH}
\crefname{thm}{Theorem}{Theorems}
\crefname{proposition}{Proposition}{Propositions}
\crefname{app}{Appendix}{Appendices}
\crefname{eq}{Equation}{Equations}
\Crefname{lemma}{Lemma}{Lemmas}
\crefname{alg}{Algorithm}{Algorithms}
\crefname{ex}{Example}{Examples}
\crefname{conj}{Conjecture}{Conjectures}
\crefname{notz}{notation}{notations}

\makeatletter
\newcommand{\xleftrightarrow}[2][]{\ext@arrow 3359\leftrightarrowfill@{#1}{#2}}
\newcommand{\xdashrightarrow}[2][]{\ext@arrow 0359\rightarrowfill@@{#1}{#2}}
\newcommand{\xdashleftarrow}[2][]{\ext@arrow 3095\leftarrowfill@@{#1}{#2}}
\newcommand{\xdashleftrightarrow}[2][]{\ext@arrow 3359\leftrightarrowfill@@{#1}{#2}}
\def\rightarrowfill@@{\arrowfill@@\relax\relbar\rightarrow}
\def\leftarrowfill@@{\arrowfill@@\leftarrow\relbar\relax}
\def\leftrightarrowfill@@{\arrowfill@@\leftarrow\relbar\rightarrow}
\def\arrowfill@@#1#2#3#4{%
  $\m@th\thickmuskip0mu\medmuskip\thickmuskip\thinmuskip\thickmuskip
   \relax#4#1
   \xleaders\hbox{$#4#2$}\hfill
   #3$%
}
\makeatother

\usepackage{geometry}
\usepackage{dynkin-diagrams}

\pgfkeys{/Dynkin diagram,
  edge length  = 1.35cm,
  root radius  = .085cm,
  label macro/.code = {\omega_{#1}},   
}
 
\definecolor{ClassicBlue}{RGB}{10,60,130}
\definecolor{ExcBlue}{RGB}{140,20,20}
\definecolor{RowGray}{gray}{0.94}

\title{Prime Fano fourfolds in classical and generalized Grassmannians} 
\author{Alessandro Frassineti}
\address{Dipartimento di Matematica \\
Dipartimento di Eccellenza 2023-2027\\
Universit\`a di Genova\\
Via Dodecaneso 35\\
16146 Genova, Italy}
\email[A.~Frassineti]{alessandro.frassineti@edu.unige.it}

\begin{document}

\begin{abstract}
We provide a complete classification of Fano fourfolds of Picard rank $1$ and index $1$ obtained as zero loci of completely reducible homogeneous vector bundles in generalized Grassmannians, i.e. Grassmannians of Dynkin type different from $A_n$. This extends the classification done by K\"uchle in standard Grassmannians. In doing so, we exhibit three new families of Fano fourfolds of Picard rank $1$ and index $1$. For each of the studied Fano fourfolds, we compute several numerical invariants, such as volume and Hodge numbers.
\end{abstract} 
\maketitle

\section{Introduction}
In any dimension, there are only finitely many deformation families of Fano manifolds \cite{KollarMori}. Therefore, the classification of Fano manifolds is one of the main open problems in algebraic geometry. A complete classification is currently known up to dimension $3$, while for dimension $4$ there are some partial results and databases collecting families of Fano fourfolds which have appeared in the literature. See for instance \cite{Kalash,BernadaraFatManTant,FaTanTufo,Fanocnrs}. 

The Picard rank of a Fano manifold coincides with the second Betti number and so is a deformation invariant. Families of Picard rank $1$ are called \textit{prime Fano manifolds} and they occupy a prominent role in the classification. In fact, one should expect that many other families are obtained from them through birational operations such as blow-ups. Another important numerical invariant is the \textit{index} of a Fano manifold, i.e. the maximal integer $k$ which divides the canonical bundle in the Picard group. 

The classification of prime Fano threefolds, which started with Fano himself, was achieved by Iskovskikh (\cite{Isko1, Isko2, IskoAntCan}). Later, Mukai introduced the \textit{vector bundle method} (\cite{mukaiBireg, mukaiVBM}) to describe all $17$ families of prime Fano threefolds as complete intersections of homogeneous vector bundles in standard Grassmannians or weighted projective spaces. 

Subsequently, Mori and Mukai in \cite{MoriMukai} completed the classification of Fano threefolds using Mori's theory of extremal rays. More recently, in \cite{FatTant3folds}, a description of all Fano threefolds as the zero locus of homogeneous vector bundles in (products of) standard Grassmannians or standard flag manifolds has been achieved.

Prime Fano fourfolds of index greater than $1$ have been classified by combining Mukai's work on Fano threefolds with the result in \cite{FujitaDelPezzo}. We now focus on prime Fano fourfolds of index $1$. Apart from classical constructions, such as complete intersections in (weighted) projective spaces and toric Fano varieties, the first systematic way to produce new prime Fano fourfolds has been the so-called \kuchle\ list \cite{kuchleList}. There, Fano fourfolds arising as complete intersections of homogeneous vector bundles in standard Grassmannians are classified, in the same spirit as Mukai's first classification. 

In this paper, we generalize this classification to the case of Grassmannians of any Dynkin type, leading to the following result:

\begin{thm}\label[theorem]{mainTh}
    Let $X=G/P$ be a rational homogeneous variety with $G$ a connected simply connected simple Lie group of Dynkin type different from $A_n$ and $P$ a maximal parabolic subgroup of $G$. Let $E$ be a completely reducible homogeneous vector bundle on $X$ such that it is globally generated. Then a general global section of $E$ defines a (smooth) prime Fano fourfold of index $1$ if and only if the pair $[X,E]$ appears in \cref{TabVlad}. 

    Among them, there are three families which are new to the literature, namely
    \[
    \left[\OGr(3,9),\cS^{\oplus 4}\right], \quad \left[\OGr(5,10)_+,\mathcal{O}(2)\oplus \mathcal{O}(1)^{\oplus 5}\right] \quad\text{and}\quad \left[\OGr(2,7),\mathrm{Sym}^2\cS\right] 
    \]
\end{thm}

Here and throughout this work, $\mathcal{S}_{(\pm)}$ denotes the spinor bundle(s) on the orthogonal Grassmannians while the tautological subbundle and quotient bundle are denoted by $\mathcal{U}$ and $\mathcal{Q}$ respectively.

\medskip
The Fano fourfolds listed in \cref{TabVlad} provide a complete list of Fano fourfolds of Picard rank $1$ and index $1$ which are zero loci of completely reducible homogeneous vector bundles in generalized Grassmannians. This follows from \cite[Theorem 5.2]{BenedettiTrivial} since general hyperplane sections of such manifolds are Calabi--Yau threefolds, which have been classified by Benedetti. In \cref{TabVlad}, we list all the Fano fourfolds obtained in this way from Benedetti's tables, using the labeling given there. 

Moreover, for each family of Fano fourfolds, we provide a list of numerical invariants, such as Hodge numbers and volume.

\begin{rmk}
    For the three new models, we study whether they provide a locally complete family. It turns out that one is not locally complete (\cref{F3}). 
    
    This model is quite interesting also from the point of view of the structure of the Hodge diamond. In fact, in cohomology of degree $3$ it admits a level one Hodge structure which we explain geometrically in \cref{sec_curve3}.
\end{rmk}

\begin{rmk}
    These three new families of Fano fourfolds extend the known classification of prime Fano fourfolds of index $1$ to $35$ families. They all satisfy $\chi(\mathcal{T}) < 0$. The computations for the other $32$ families have been carried out in \cite{arXiv:2606.13466}. It follows that none of these varieties satisfies Bott vanishing.
\end{rmk}

\begin{rmk}
    In \cref{TabVlad}, there are three families in $\OGr(3,8)=\mathrm{SL}(8)/P_{\{1,2\}}$ even though in this case the parabolic subgroup is not maximal. We decided to include them in the classification as well since $\OGr(3,8)$ is commonly regarded as a Grassmannian. Interestingly, the family 
    \[
    \left[\OGr(3,8), \bigwedge^2\cU^\vee\oplus  \mathcal{S}_-\oplus \mathcal{S}_+\right]
    \]
    is also new to the literature. In this family, the Picard rank is $2$. In \cref{sec_greater1}, we prove that a fourfold in this family is the blow-up a the intersection of two quadrics along a Del Pezzo surface of degree $6$.
\end{rmk}

These results shows a phenomenon which does not occur in dimension $3$, namely that there are prime Fano fourfolds that cannot be realized as complete intersections in (products of) \emph{just} standard Grassmannians. 

In fact, homogeneous varieties of Dynkin type different from $A_n$ do appear in the list of Mukai models of prime Fano threefolds (e.g. $\OGr(5,10)_+$ or $\SGr(3,6)$); however, it is still possible to realize all of those families with models in standard Grassmannians.

Although it is likely that only a small percentage of Fano manifolds (in each dimension) have this remarkable feature, this shows that the \textit{non-$A_n$} Dynkin type cannot be avoided in the classification.

\subsection*{A remark on the notation}
Throughout this work, we will use the following non-standard notation. Let $E$ be a vector bundle over a projective variety $X$. We denote by $[X,E]$ the family of subvarieties of $X$ obtained as the zero locus of a general section $s\in H^0(X,E)$. These varieties are all deformation equivalent but in general not isomorphic. Sometimes we refer to $[X,E]$ as a \textit{model} for any $Y \in [X,E]$.

In some special cases, however, one can prove that a general section provides the same variety $Y$. If this is the case, we write $Y= (X,E)$, without specifying the section. 

For instance, when $X=G/P$ is a rational homogeneous variety and $E$ is a homogeneous vector bundle on $X$, by Bott's theorem, we have an action of $G$ on $\mathbb{P}H^0(X,E)$. If this action is \textit{prehomogeneous}, which means that it admits an open orbit, a general global section $s$ will define the same variety up to the action of $G$. In particular cases, these prehomogeneous actions are classified, see for instance \cite{sato1977classification} for the case of spinor representations.

\section{Outline of the proof}
The proof of the assertions in \cref{mainTh} is divided into the different sections of this work. In \cref{sec_isomorphisms}, we list and prove a number of exceptional isomorphisms between zero loci of homogeneous vector bundles which are needed to reduce the classification in \cref{TabVlad} to the families which are actually distinct. This reduction is done in \cref{sec_classification}. 

In \cref{NewModels}, we discuss in detail the three new models presented in \cref{mainTh}. Finally, in \cref{sec_greater1}, we address those Fano fourfolds of Picard rank greater than $1$ encountered in the classification.

The rest of this section is devoted to the computation of numerical invariants for Fano fourfolds. These are fundamental for two complementary reasons. Firstly, they ensure that the families we produced are distinct and new to the literature. In addition, they allow us to spot different families which are actually the same, leading to the results contained in \cref{sec_isomorphisms}.
\subsection{Hodge numbers and invariants}\label[section]{sec_invariants}
Let $X= G/P$ be a rational homogeneous variety, where $G$ is a semi-simple complex Lie group and $P$ a parabolic subgroup. Let $Y$ be the zero locus of a general section $s\in H^0(X,E)$, where $E$ is a globally generated homogeneous vector bundle over $X$. In particular, by Bertini's theorem, $Y$ is smooth of codimension in $X$ equal to the rank of $E$. We describe the general strategy we follow to compute invariants for our Fano fourfolds, and more generally for any variety $Y$ obtained in this way. The list of deformation invariants for a Fano fourfold $Y$ that we are able to compute is: 
\begin{itemize}
    \item the index, i.e. the maximal integer $k$ such that $\frac{1}{k}K_Y\in H^{1,1}_{\mathbb{Z}}(Y) = \mathrm{Pic}(Y)$;
    \item Hodge numbers and, in particular, the Picard rank;
    \item the volume, which is by definition equal to $(-K_Y)^4$;
    \item the dimension of the space of global sections $h^0(Y,-K_Y) = \chi(Y,-K_Y)$ ;
    \item the Euler characteristic of the tangent bundle $\chi(Y,\mathcal{T}_Y)$.
\end{itemize}

Since the normal bundle of the inclusion $Y \hookrightarrow X$ is $E_{|Y}$, the canonical bundle of $Y$ can be computed by adjunction as 
\[
K_Y = (K_X + \mathsf{c}_1(E))_{|Y}.
\]
Moreover, the cotangent sequence is given by 
    \[
    0\longrightarrow E^\vee_{|Y} \longrightarrow \Omega_{X|Y} \longrightarrow\Omega_Y\longrightarrow 0
    \]
and its $p$-th exterior power is 
\begin{align*}
0&\longrightarrow \mathrm{Sym}^pE_{|Y}^\vee \longrightarrow \mathrm{Sym}^{p-1}E_{|Y}^\vee\otimes\Omega_{X|Y} \longrightarrow \cdots  \hspace*{3cm} \\ 
\hspace*{3cm}
\cdots &\longrightarrow  E_{|Y}^\vee \otimes \Omega_{X|Y}^{p-1}  \longrightarrow \Omega_{X|Y}^{p} \longrightarrow        \Omega_Y^{p} \longrightarrow 0.
\end{align*}
In order to compute the Hodge numbers $h^{p,q}(Y)$, we use the Koszul complex defined by the vector bundle section $s$, and we use Bott's theorem to compute the cohomology of $$\left(\mathrm{Sym}^j\mathcal{E}^\vee \otimes \Omega_{X}^{p-j}\right)_{|Y}.$$ In many cases, in the induced long exact sequences in cohomology, enough terms vanish for the computation to be successful and one obtains all the Hodge numbers of $Y$. Sometimes, however, the boundary maps are not trivial and ad-hoc arguments are needed (see e.g. the proof of \cref{F3}).

In principle, starting from the tangent sequence
    \[
    0\longrightarrow \mathcal{T}_Y \longrightarrow \mathcal{T}_{X|Y} \longrightarrow E_{|Y}\longrightarrow 0
    \]
    instead, one can compute the Hochschild numbers of $Y$ in a similar way. However, since they are not invariant under deformations in general, we prefer to provide only the holomorphic Euler characteristic of the tangent bundle $\chi(\mathcal{T}_Y)$. In some cases, we are also able to compute the actual cohomology of the tangent bundle (see e.g. \cref{F2}). A special Hochschild number we provide is the dimension of the anti-canonical global sections $h^0(Y,-K_Y)$.

    Lastly, the volume of the Fano fourfold is computed using the following 
    \begin{lemma}
    Let $Y$ be a Fano fourfold. Its volume is equal to
    \[
        c_1(-K_Y)^4 = \chi(Y,-2K_Y) -3\chi(Y,-K_Y)+2.
    \]
    \end{lemma}
    \begin{proof}
        This is a straightforward application of Hirzebruch--Riemann--Roch.
        Since $Y$ is Fano, we have
        \[
        1 = \chi(\mathcal{O}_Y) = \mathsf{td}_4(Y).
        \]
        Denote by $c_i := c_i(-K_Y)$ the $i$-th Chern class of the anticanonical bundle of $Y$. The other Todd classes are given by 
        \[
        \mathsf{td}_1(Y) = \frac{c_1}{2} \quad \mathsf{td}_2(Y) = \frac{c_1^2 + c_2}{12} \quad \mathsf{td}_3(Y) = \frac{c_1c_2}{24}.
        \]
        Since the Chern characters of $-K_Y$ and $-2K_Y$ are respectively
        \begin{align*}
            \mathsf{ch}(-K_Y) &= 1 + c_1 + \frac{c_1^2}{2}+ \frac{c_1^3}{6}+ \frac{c_1^4}{24}\\
            \mathsf{ch}(-2K_Y) &= 1 + 2c_1 + 2c_1^2+ \frac{4c_1^3}{3}+ \frac{2c_1^4}{3},
        \end{align*}
        by Hirzebruch--Riemann--Roch, we get
        \begin{align*}
            \chi(-K_Y) &= 1 + \frac{c_1^2c_2}{12}+ \frac{c_1^4}{6}\\
            \chi(-2K_Y) &= 1 + \frac{c_1^2c_2}{4}+ \frac{3c_1^4}{2}.
        \end{align*}
        The result follows.

    \end{proof}
    
    In particular, all three families appearing in \cref{NewModels} are Fano fourfolds of index $1$, since their volumes are not divisible by the fourth power of any integer greater than $1$.

\subsection{Local completeness}
Let $Y$ be the zero locus of a section $s\in H^0(X,E)$, where $E$ is a vector bundle over a smooth projective variety $X$. Suppose that the codimension of $Y$ in $X$ is equal to the rank of $E$. We say that the family of subvarieties defined by varying the section $s$ is \emph{locally complete} if any small deformation of $Y$ can be obtained as the zero locus of a global section of $E$. The following vanishings provide a sufficient condition for local completeness (cf. \cite{wehler}):
\begin{enumerate}[(a)]
    \item \label{w1} $H^1(Y, \mathcal{T}_{X|Y}) = 0$;
    \item \label{w2} $H^1(X, \mathcal{I}_Y\otimes E) = 0$.
\end{enumerate}
Here, $\mathcal{I}_Y$ denotes the ideal sheaf of $\iota: Y \hookrightarrow X$, i.e., we have the short exact sequence
\[
0\longrightarrow \mathcal{I}_Y \longrightarrow \mathcal{O}_X \longrightarrow \iota_*\mathcal{O}_Y\longrightarrow 0.
\]
In fact, starting from the normal sequence
\begin{equation}\label{normalSequence}
    0\longrightarrow \mathcal{T}_Y \longrightarrow \mathcal{T}_{X|Y} \longrightarrow E_{|Y}\longrightarrow 0,
\end{equation}
we can consider the following piece of the long exact sequence in cohomology:
\begin{equation}\label{parametersSequence}
\begin{tikzcd}
{H^0(Y,\mathcal{T}_Y)} \arrow[r] & {H^0(Y,\mathcal{T}_{X|Y})} \arrow[r] & {H^0(Y,E_{|Y})} \arrow[r, "\varphi"]    & {H^1(Y,\mathcal{T}_Y)} \arrow[r] & {H^1(Y,\mathcal{T}_{X|Y})} \\
            &                                  &                                       {H^0(X,E)} \arrow[u, "\psi"] \arrow[ru]                                  &                           
\end{tikzcd}
\end{equation}
By condition \ref{w2}, the map $\psi$ is surjective, while by condition \ref{w1}, $\varphi$ is surjective. Therefore, any element of $H^1(Y,\mathcal{T}_Y)$ comes from a global section of $E$.

If we consider the Koszul complex associated to the section of $E$ which defines $Y$ and we twist it by $E$, we obtain the following sufficient condition for \ref{w2}:
\begin{enumerate}[(b')]
    \item \label{w2'} $H^p\left(X, \bigwedge^pE^\vee \otimes E\right) = 0$ for $p> 0$.
\end{enumerate}
Therefore, in order to prove local completeness of our models, it is enough to check conditions \ref{w1} and \ref{w2'}.

\section{Unexpected isomorphisms}\label[section]{sec_isomorphisms}
In this section, we recall and prove some isomorphisms between zero loci of homogeneous vector bundles. Some of them already appeared in \cite{ManivelFrassineti}, and for those we will just recall the statement.

\begin{proposition}\label[proposition]{OG27}
    $(\OGr(2,7),\mathcal{S})\simeq \bigslant{G_2}{P_2}$ is a Fano five-fold of index $3$.
\end{proposition}
\begin{proof}
See \cite[Proposition 2]{ManivelFrassineti}.
\end{proof}

\begin{proposition}\label[proposition]{OG29}
    $(\OGr(2,9),\mathcal{S})\simeq \OGr(2,7)$ is a Fano seven-fold of index $4$.
\end{proposition}
\begin{proof}
See \cite[Proposition 6]{ManivelFrassineti}.
\end{proof}

\begin{rmk}
    The restriction of $\mathcal{U}^\vee$ to $\OGr(2,7) \subseteq \OGr(2,9)$ is again homogeneous under the action of $\mathrm{Spin}(7)$. This follows from the proof of \cite[Proposition 25]{sato1977classification}, where they constructed a non-standard embedding 
    \[
    \mathrm{Spin}(7) \hookrightarrow \mathrm{Spin}(9).
    \]
    In fact, the bundle $\mathcal{U}^\vee$ restricts to the spinor bundle $\cS$ on $\OGr(2,7)$. This can be checked by computing the space of global sections of the restriction, which is $8$-dimensional. But $\cS$ is the only rank $2$ homogeneous vector bundle on $\OGr(2,7)$ with determinant $\mathcal{O}(1)$ and $8$-dimensional space of global sections.
\end{rmk}

\begin{proposition}\label[proposition]{SG26}
    $(\SGr(2,6),(\cU^\perp/\cU)(1))\simeq \bigslant{G_2}{P_2}$ is a Fano five-fold of index $3$.
\end{proposition}
\begin{proof}
    Let $\omega$ be an anti-symplectic form on $V_6$ and let $V_7 = V_6 \oplus \langle f_7\rangle$. We can extend $\omega$ to $V_7$ by imposing that $f_7$ lies in the kernel. The projection from $f_7$ gives a rational map
    \[
    \varphi:\SGr(2,7) \dashrightarrow \SGr(2,6), \quad \langle q_1+\lambda f_7, q_2 + \mu f_7\rangle \longmapsto \langle q_1,q_2\rangle.
    \]
    Let $[\Omega] \in H^0\left(\SGr(2,6),(\cU^\perp/\cU)(1)\right)$ be a general section. Recall that the latter is equal to 
    \[
    \bigwedge^{\langle 3 \rangle}V_6 := \bigslant{\bigwedge^3V_6}{\omega \wedge V_6},
    \]
    hence $\Omega \in \bigwedge^3V_6 \simeq \bigwedge^3V_6^\vee$. Since $$H^0\left(\SGr(2,7),\mathcal{Q}^\vee(1)_{|\SGr}\right)\simeq H^0(\Gr(2,7),\mathcal{Q}^\vee(1)) = \bigwedge^3V_7^\vee,$$ we get a natural element $\Omega^+ := \Omega + f_7^*\wedge \omega$. If $P\in \SGr(2,7)$ annihilates $\Omega^+$, then 
    \[
    0 = \Omega^+(p_1,p_2,-) = \Omega(q_1,q_2,-) + \lambda\omega(q_2,-) - \mu \omega(q_1,-) + \omega(q_1,q_2)f_7^*(-).
    \]
    In particular, since $\langle q_1,q_2\rangle$ is isotropic and $[\Omega] = \Omega + \omega \wedge V_6$, we deduce that $\varphi (P) $ annihilates $[\Omega]$. If moreover $P$ lies in the base locus of $\varphi$, then it can be written as $\langle q_1, f_7\rangle$ and we get 
    \[
    0 = \Omega^+(q_1,f_7,-) = - \omega(q_1,-),
    \]
    which is a contradiction since $\omega$ is non-degenerate. Let $X := \mathcal{Z}([\Omega])$ and $X':= \mathcal{Z}(\Omega^+)$. We have a map
    \[
    \overline{\varphi}: X' \longrightarrow X.
    \]
    If we take an element in $X'$ with $\lambda = \mu = 0$, i.e. $P = \langle q_1,q_2\rangle$, we obtain 
    \[
    0 = \Omega(q_1,q_2,-) + \omega(q_1,q_2)f_7^*(-).
    \]
    Thus, $\Omega(q_1,q_2,-) = \omega(q_1,q_2) =0$ and it follows that the map 
    \[
    \psi:X \longrightarrow X', \quad \langle q_1, q_2 \rangle \longmapsto \langle q_1,q_2\rangle.
    \]
    is the inverse of $\overline{\varphi}$.
\end{proof}

\begin{proposition}\label[proposition]{OG210}
$(\OGr(2,10), \cS_+\oplus\cS_+) \simeq \bigslant{G_2}{P_2}$ is a Fano five-fold of index $3$.
\end{proposition}
\begin{proof}
See \cite[Proposition 12]{ManivelFrassineti}.
\end{proof}

\begin{proposition}\label[proposition]{fl134}
$(\OGr(2,10), \cS_+\oplus\cS_-)=\mathrm{Fl}(1,3,4)$ is a Fano five-fold of index $3$.
\end{proposition}
\begin{proof}
See \cite[Proposition 16]{ManivelFrassineti}.
\end{proof}

\begin{proposition}\label[proposition]{OG214}
    $(\OGr(2,14), \cS_+) \simeq G_2/P_2\cup G_2/P_2$ is a disjoint union of two Fano five-folds of index $3$.
\end{proposition}
\begin{proof}
    See \cite[Proposition 26]{ManivelFrassineti}.
\end{proof} 

\begin{proposition}\label[proposition]{SG28}
    Consider on $\SGr(2,8)$ the irreducible vector bundle $E$ with highest weight $\omega_4=[0,0,0,1]$ in the root system $C_4$. 
    
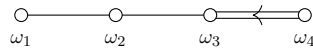
\[
\begin{tikzpicture}[baseline=-3pt]
    \dynkin[label, edge length=1.25cm]{C}{oooo}
  \end{tikzpicture}
\]    
    Then $E$ satisfies the following equality:
    \[
    \left(\bigwedge^2(\cU^\perp/\cU)\right)(1) \simeq \mathcal{O}(1) \oplus E.
    \]
    Moreover, $$(\SGr(2,8), E)  \simeq (\OGr(2,8), \mathrm{Sym}^2\cS_+)\simeq (\OGr(2,8), \mathrm{Sym}^2\cU^\vee)$$ is a Fano sixfold of index $2$.
\end{proposition}
\begin{proof}
    Let $V_8$ be a complex vector space endowed with a symmetric form $q$ and an anti-symmetric form $\omega$, both non-degenerate. In particular, we have an action on $V_8$ of the symplectic group $\mathrm{Sp}(8)$ and of the orthogonal group $\mathrm{Spin}(8)$. The induced representation $\bigwedge^4V_8$ decomposes into irreducible ones under the first action as
    \[
    \bigwedge^4V_8\simeq \bigwedge^{\langle 4 \rangle}V_8 \oplus \bigwedge^{\langle 2 \rangle}V_8 \oplus \mathbb{C} 
    \]
    and, under the second action, as
    \[
    \bigwedge^4V_8\simeq \bigslant{\left(\mathrm{Sym}^2\Delta_+\right)}{\mathbb{C}} \oplus  \bigslant{\left(\mathrm{Sym}^2\Delta_-\right)}{\mathbb{C}}
    \]
    where $\Delta_\pm$ are the half-spinor representations of $\mathrm{Spin}(8)$. Using Bott's theorem, we obtain
    \[
    H^0(\SGr(2,8),E) \simeq \bigwedge^{\langle 4 \rangle}V_8.
    \]
    In particular, choosing a general global section of $E$ together with a symplectic form yields an element of $\bigwedge^4V_8$, which in turn gives a global section of the two spinor bundles on $\OGr(2,8)$. Using the triality property of $D_4$, one can realize the two sections as a quadratic form and a global section of $\mathrm{Sym}^2\cS_+$. Thus, we get the isomorphism
    \[
    (\SGr(2,8), E) \simeq (\OGr(2,8), \mathrm{Sym}^2\cS_+).
    \]
    Again using triality at the level of homogeneous vector bundles, we get $$(\OGr(2,8), \mathrm{Sym}^2\cS_+)\simeq (\OGr(2,8), \mathrm{Sym}^2\cU^\vee).$$    
\end{proof}

\begin{lemma}\label[lemma]{splitQ'1}
    Let $0\neq f\in H^0(\Gr(k,V_n),\cU^\vee) \simeq V_n^\vee$ and let $V_{n-1}\subset V_n$ be the corresponding hyperplane. Denote by $i:\Gr(k,n-1) \hookrightarrow \Gr(k,n)$ the induced embedding. Then $i^*\mathcal{Q}_{n-k} \simeq \mathcal{Q}_{n-1-k} \oplus \mathcal{O}$. In particular, $$H^0\left(\Gr(k,n),\mathcal{Q}^\vee_{n-k}(1) \right) \simeq H^0\left(\Gr(k,n-1),\mathcal{Q}^\vee_{n-1-k}(1) \oplus \mathcal{O}(1)\right).$$
\end{lemma}
\begin{proof}
    We have a short exact sequence
    \[
    0\longrightarrow \mathcal{Q}_{n-1-k} \longrightarrow i^*\mathcal{Q}_{n-k} \longrightarrow \bigslant{V_n}{V_{n-1}}\otimes \mathcal{O}\longrightarrow 0
    \]
    which at the fiber level reads as
    \[
    0\longrightarrow \bigslant{V_{n-1}}{W} \longrightarrow \bigslant{V_n}{W} \longrightarrow \bigslant{V_n}{V_{n-1}}\longrightarrow 0.
    \]
    Since $\mathrm{Ext}^1(\mathcal{O},\mathcal{Q}_{n-1-k}) = 0$, the above sequence splits.
\end{proof}

\begin{lemma}\label[lemma]{splitQ'2}
    Let $g\in H^0\left(\Gr(k,V_n),\bigwedge^2\cU^\vee\right) \simeq \bigwedge^2V_n^\vee$ be a non-degenerate symplectic form. Denote by $j:\SGr(k,n) \hookrightarrow \Gr(k,n)$ the induced embedding. Then $$H^0\left(\Gr(k,n),\mathcal{Q}^\vee_{n-k}(1) \right) \simeq H^0\left(\SGr(k,n),(\cU^\perp/\cU)(1) \oplus \bigwedge^{k-1}\cU^\vee\right).$$
\end{lemma}
\begin{proof}
    We have a short exact sequence
    \[
    0\longrightarrow \cU^\perp/\cU \longrightarrow j^*\mathcal{Q}_{n-k} \longrightarrow \cU^\vee\longrightarrow 0
    \]
    which at the fiber level reads as
    \[
    0\longrightarrow \bigslant{W^\perp}{W} \longrightarrow \bigslant{V_n}{W} \longrightarrow W^\vee\longrightarrow 0.
    \]
    Here $W$ denotes an isotropic $k$-plane and the second map is $v + W \longmapsto g(v,-)$. 
    Taking the dual of the above sequence and twisting by $\mathcal{O}(1)$, we get
    \[
    0\longrightarrow \bigwedge^{k-1}\cU^\vee \longrightarrow j^*\mathcal{Q}^\vee_{n-k}(1) \longrightarrow (\cU^\perp/\cU)(1)\longrightarrow 0.
    \]
    Since $H^1\left(\SGr(k,n), \bigwedge^{k-1}\cU^\vee\right)=0$ by Bott's theorem, the result follows.
\end{proof}

\section{The classification}\label[section]{sec_classification}
The subject of our research is Fano fourfolds of Picard rank $1$ and index $1$ described as zero loci of completely reducible homogeneous vector bundles in classical and generalized Grassmannians. Using the isomorphisms provided in \cref{sec_isomorphisms}, we are able to identify the distinct families of Fano fourfolds among the ones in \cref{TabVlad}.
We work out these identifications in the following series of lemmas.

\begin{lemma}
    The Fano fourfolds of type \textup{\ref{item_sb2}}, \textup{\ref{item_ox3}}, \textup{\ref{item_ob5}}, \textup{\ref{item_ob9}}, \textup{\ref{item_ow7}} and \textup{\ref{item_g2}} provide the same family as $\left[G_2/P_2,\mathcal{O}(2)\right]$. It corresponds to the family \textup{\cite[(b10)]{kuchleList}}. We recall here the invariants
    $$\xymatrix@1@=3pt@M=0pt{&&&&1&&&&\\&&&0&&0&&& \\ &&0&&1&&0&& \\ &0&&0&&0&&0& \\ 0&&14&&100&&14&&0}$$
    \vspace{0.5pt}
    \[
    (-K)^4 = 36 \quad h^0(-K) = 14 \quad \chi(\mathcal{T}) = -62.
    \]
\end{lemma}
\begin{proof}
    By \cref{SG26}, we have that $\ref{item_sb2} \simeq \left[G_2/P_2,\mathcal{O}(2)\right]$. By \cref{OG27} and by triality, we have $\ref{item_ow7} \simeq \ref{item_ox3} \simeq\left[G_2/P_2,\mathcal{O}(2)\right]$. By \cref{OG210}, $\ref{item_ob5} \simeq\left[G_2/P_2,\mathcal{O}(2)\right]$ and finally by \cref{OG214}, $\ref{item_ob9} \simeq\left[G_2/P_2,\mathcal{O}(2)\right]$. The family \cite[(b10)]{kuchleList} is given by a quadratic section of $\left(\OGr(2,7), \mathcal{Q}^\vee(1)\right)$, which is a classical geometric realization of $G_2/P_2$.
\end{proof}

\begin{lemma}
    The Fano fourfold of type \textup{\ref{item_sb5}}, i.e.  $\left[\SGr(2,8), (\cU^\perp/\cU)(1)\oplus \cU^\vee\oplus \mathcal{O}(1)\right]$, corresponds to the family \textup{\cite[(b11)]{kuchleList}}. We recall here the invariants
    $$\xymatrix@1@=3pt@M=0pt{&&&&1&&&&\\&&&0&&0&&& \\ &&0&&1&&0&& \\ &0&&0&&0&&0& \\ 0&&2&&31&&2&&0}$$
    \vspace{0.5pt}
    \[
    (-K)^4 = 57 \quad h^0(-K) = 18 \quad \chi(\mathcal{T}) = -28.
    \]
\end{lemma}
\begin{proof}
    The family \cite[(b11)]{kuchleList} is given by $\left[\OGr(2,8), \mathcal{Q}^\vee(1)\oplus \mathcal{O}(1)^{\oplus 2}\right]$. By \cref{splitQ'2}, it corresponds to $\ref{item_sb5}$.
\end{proof}

\begin{lemma}
    The Fano fourfold of type \textup{\ref{item_sc1}}, i.e.  $\left[\SGr(3,8), (\cU^\perp/\cU)(1)\oplus \bigwedge^2\cU^\vee \oplus \cU^\vee \right]$, corresponds to the family \textup{\cite[(c5)]{kuchleList}}. We recall here the invariants
    $$\xymatrix@1@=3pt@M=0pt{&&&&1&&&&\\&&&0&&0&&& \\ &&0&&1&&0&& \\ &0&&0&&0&&0& \\ 0&&1&&24&&1&&0}$$
    \vspace{0.5pt}
    \[
    (-K)^4 = 66 \quad h^0(-K) = 20 \quad \chi(\mathcal{T}) = -25.
    \]
\end{lemma}

\begin{proof}
    The family \cite[(c5)]{kuchleList} is given by
    $\left[\Gr(3,7), \mathcal{Q}^\vee(1) \oplus \mathcal{O}(1) \oplus \bigwedge^2 \cU^\vee\right]$. We apply \cref{splitQ'1} and \cref{splitQ'2} to $i:\Gr(3,7) \hookrightarrow \Gr(3,8)$ and $j:\SGr(3,8) \hookrightarrow \Gr(3,8)$ respectively. The result follows by taking a general global section of $\mathcal{Q}^\vee(1) \oplus \cU^\vee \oplus \bigwedge^2\cU^\vee$ on $\Gr(3,8)$.
\end{proof}

\begin{lemma}
    The Fano fourfolds of type \textup{\ref{item_ob2}} and \textup{\ref{item_ow14}} provide the same family as $\left[\OGr(2,7),\mathcal{O}(1)^{\oplus 3}\right]$. It corresponds to the family \textup{\cite[(b8)]{kuchleList}}. We recall here the invariants
    $$\xymatrix@1@=3pt@M=0pt{&&&&1&&&&\\&&&0&&0&&& \\ &&0&&1&&0&& \\ &0&&0&&0&&0& \\ 0&&3&&38&&3&&0}$$
    \vspace{0.5pt}
    \[
    (-K)^4 = 56 \quad h^0(-K) = 18 \quad \chi(\mathcal{T}) = -33.
    \]
\end{lemma}
\begin{proof}
    By \cref{OG29}, we have that $\ref{item_ob2} \simeq \left[\OGr(2,7),\mathcal{O}(1)^{\oplus 3}\right]$. By triality, we have $\ref{item_ow14} \simeq \left[\OGr(2,8),\cU^\vee \oplus \mathcal{O}(1)^{\oplus 3}\right]\simeq \left[\OGr(2,7),\mathcal{O}(1)^{\oplus 3}\right]$. Clearly, they all correspond to \cite[(b8)]{kuchleList}, i.e. $$\left[\Gr(2,7),\mathrm{Sym}^2\cU^\vee\oplus \mathcal{O}(1)^{\oplus 3}\right].$$
\end{proof}

\begin{lemma}
    The Fano fourfold of type \textup{\ref{item_oy2}}, i.e. $\left[\OGr(4,8)_+, \mathcal{O}(2) \oplus \mathcal{O}(3) \right]$, is a complete intersection in $\mathbb{P}^7$ of a cubic and two quadrics. We recall here the invariants
    $$\xymatrix@1@=3pt@M=0pt{&&&&1&&&&\\&&&0&&0&&& \\ &&0&&1&&0&& \\ &0&&0&&0&&0& \\ 0&&42&&236&&42&&0}$$
    \vspace{0.5pt}
    \[
    (-K)^4 = 12 \quad h^0(-K) = 8 \quad \chi(\mathcal{T}) = -108.
    \]
\end{lemma}
\begin{proof}
    By triality, $\OGr(4,8)_+\simeq \left( \mathbb{P}^7, \mathcal{O}(2)\right)$.
\end{proof}

\begin{lemma}
    The Fano fourfolds of type \textup{\ref{item_oz1}}, \textup{\ref{item_oz2}} and \textup{\ref{item_g1}} provide the same family as $\left[G_2/P_1,\mathcal{O}(4)\right]$, which is a complete intersection in $\mathbb{P}^6$ of a quartic and a quadric. We recall here the invariants
    $$\xymatrix@1@=3pt@M=0pt{&&&&1&&&&\\&&&0&&0&&& \\ &&0&&1&&0&& \\ &0&&0&&0&&0& \\ 0&&77&&394&&77&&0}$$
    \vspace{0.5pt}
    \[
    (-K)^4 = 8 \quad h^0(-K) = 7 \quad \chi(\mathcal{T}) = -160.
    \]
\end{lemma}
\begin{proof}
    A standard geometric realization of $G_2/P_1$ is provided by a quadric fivefold, which by triality is the same as $\left(\OGr(4,8)_+, \mathcal{O}(1)\right)$. On the other hand, $\left(\OGr(5,10)_+,\cU(1))\right) \simeq G_2/P_1$ since a Fano $n$-fold of index $n$ is isomorphic to an $n$-dimensional quadric.
\end{proof}

\begin{lemma}
    The Fano fourfold of type \textup{\ref{item_oz5}}, i.e. $\left[\OGr(6,12)_+, \cU(1) \oplus \mathcal{O}(1)^{\oplus 5}\right]$, corresponds to the family \textup{\cite[(c1)]{kuchleList}}. We recall here the invariants
    $$\xymatrix@1@=3pt@M=0pt{&&&&1&&&&\\&&&0&&0&&& \\ &&0&&1&&0&& \\ &0&&0&&0&&0& \\ 0&&5&&52&&5&&0}$$
    \vspace{0.5pt}
    \[
    (-K)^4 = 42 \quad h^0(-K) = 15 \quad \chi(\mathcal{T}) = -40.
    \]
\end{lemma}
\begin{proof}
    The family \cite[(c1)]{kuchleList} is given by $\left[\Gr(3,6),\mathcal{O}(1)^{\oplus 5}\right]$. It corresponds to our model by the isomorphism $\left(\OGr(6,12)_+,\cU(1)\right) \simeq \Gr(3,6)$, proved in Benedetti's PhD thesis. Another proof can be found in \cite[Proposition 28]{ManivelFrassineti}. In fact, by the outer automorphism which exchanges the two antennas of $D_6$, we have $\left(\OGr(6,12)_+,\cU(1)\right) \simeq \left(\OGr(6,12)_-,\cS_+\right)$.
\end{proof}

\section{Three new models}\label[section]{NewModels}

\begin{proposition}
    The Fano fourfold of type \textup{\ref{item_ox6}}, i.e. $\left[\OGr(3,9),\mathcal{S}^{\oplus 4}\right]$, provides a new prime Fano fourfold of index $1$. This family is locally complete and the invariants are the following
    $$\xymatrix@1@=3pt@M=0pt{&&&&1&&&&\\&&&0&&0&&& \\ &&0&&1&&0&& \\ &0&&0&&0&&0& \\ 0&&0&&8&&0&&0}$$
    \vspace{0.5pt}
    \[
    (-K)^4 = 116 \quad h^0(-K) = 30 \quad \chi(\mathcal{T}) = -12.
    \]
\end{proposition}
\begin{proof}
    All the invariants are computed as described in \cref{sec_invariants}. In particular, since the volume $116= 4\cdot 29$ is not divisible by the fourth power of any integer greater than $1$, the anti-canonical bundle is not divisible in the Picard group. In particular, the index is $1$. Local completeness is obtained by checking conditions \ref{w1} and \ref{w2'}. 
\end{proof}

\begin{proposition}\label[proposition]{F2}
    The Fano fourfolds of type \textup{\ref{item_oy3}} and \textup{\ref{item_oz4}} provide the same family as $$\left[\OGr(5,10)_+,\mathcal{O}(2) \oplus \mathcal{O}(1)^{\oplus 5}\right].$$ This is a new locally complete family of dimension $65$ of prime Fano fourfolds of index $1$, whose invariants are the following
    $$\xymatrix@1@=3pt@M=0pt{&&&&1&&&&\\&&&0&&0&&& \\ &&0&&1&&0&& \\ &0&&0&&0&&0& \\ 0&&16&&112&&16&&0}$$ 
    \vspace{0.5pt}
    \[
    (-K)^4 = 24 \quad h^0(-K) = 11 \quad \chi(\mathcal{T}) = -65.
    \]
\end{proposition}
\begin{proof}
    The two families coincide by the classical isomorphism $\OGr(4,9) \simeq \OGr(5,10)_+$. Let $Y$ be the zero locus of a general global section of $E:= \mathcal{O}(2) \oplus \mathcal{O}(1)^{\oplus 5}$. In this case, it is possible to explicitly compute the dimension of $H^1(Y,\mathcal{T}_Y)$ by proving the vanishing of $H^0(Y,\mathcal{T}_Y)$ as is done in \cite[Proposition 2.1]{ManivelKuchle}. Since $Y$ is a Fano fourfold of index $1$, we have $\mathcal{T}_Y\simeq \Omega_Y^3(1)$. By the cotangent sequence, the third wedge power of $\Omega_Y$ twisted by $\mathcal{O}(1)$ admits a resolution 
    \begin{align*}
        0&\longrightarrow \left(\mathrm{Sym}^3E^\vee\right)_{|Y} (1)\longrightarrow \left(\mathrm{Sym}^2E^\vee\otimes\Omega_\OGr \right)_{|Y}(1) \longrightarrow \\ 
        &\longrightarrow \left(E^\vee\otimes\Omega^2_\OGr \right)_{|Y}(1) \longrightarrow \Omega^3_{\OGr|Y}(1) \longrightarrow \Omega^3_Y(1) \longrightarrow 0
    \end{align*}
    We obtain three short exact sequences
    \begin{align*}
        0\longrightarrow \left(\mathrm{Sym}^3E^\vee\right)_{|Y} (1)&\longrightarrow \left(\mathrm{Sym}^2E^\vee\otimes\Omega_\OGr \right)_{|Y}(1) \longrightarrow K_1 \longrightarrow 0\\ 
        0\longrightarrow K_1 &\longrightarrow \left(E^\vee\otimes\Omega^2_\OGr \right)_{|Y}(1) \longrightarrow K_0 \longrightarrow 0\\
        0\longrightarrow K_0 &\longrightarrow \Omega^3_{\OGr|Y}(1) \longrightarrow \Omega^3_Y(1) \longrightarrow 0.
    \end{align*}
    Our goal is to prove that $H^0(Y,\Omega^3_Y(1))=0$. It suffices to prove the following four vanishings:
    \begin{enumerate}[(v1)]
        \item $H^k\left(Y,\Omega^{3-k}_{\OGr|Y}(1-k)\right)=0$ for $k = 0,1,2,3$;
        \item $H^k\left(Y,\Omega^{3-k}_{\OGr|Y}(-k)\right)=0$ for $k = 1,2,3$;
        \item $H^k\left(Y,\Omega^{3-k}_{\OGr|Y}(-1-k)\right)=0$ for $k = 2,3$;
        \item $H^3\left(Y,\mathcal{O}_Y(-5)\right)=0$.
    \end{enumerate} 
    As usual, with the Koszul resolution, we can compute the cohomology of all the restricted bundles and check the desired vanishings.
\end{proof}

\begin{proposition}\label[proposition]{F3}
    The Fano fourfolds of type \textup{\ref{item_sb4}}, \textup{\ref{item_ob1}}, \textup{\ref{item_ox5}}, \textup{\ref{item_ow10}} and \textup{\ref{item_ow12}} provide the same family as $\left[\OGr(2,7),\mathrm{Sym}^2\mathcal{S}\right]$. This is a new family of Fano fourfolds, whose invariants are the following:
    $$\xymatrix@1@=3pt@M=0pt{&&&&1&&&&\\&&&0&&0&&& \\ &&0&&1&&0&& \\ &0&&3&&3&&0& \\ 0&&0&&8&&0&&0}$$
    \vspace{0.5pt}
    \[
    (-K)^4 = 72 \quad h^0(-K) = 21 \quad \chi(\mathcal{T}) = -14.
    \]
    However, this family is not locally complete.
\end{proposition}
\begin{proof}
    By \cref{OG29}, $\ref{item_ob1}\simeq \ref{item_ox5} = \left[\OGr(2,7),\mathrm{Sym}^2\cS\right] $. By triality, $$\ref{item_ow10} \simeq \ref{item_ow12} \simeq \left[\OGr(2,8),\mathrm{Sym}^2\cS_+\oplus \mathcal{U}^\vee\right].$$ 
    Finally, by \cref{SG28}, we have $\ref{item_sb4}\simeq \ref{item_ow10}$.
    
    Let $Y$ be the zero locus of a general global section of $E:= \mathrm{Sym}^2\cS$. Regarding the invariants, in this case the algorithmic computation of the Hodge numbers leaves the possibility that $h^{3,1}(Y) \neq 0$. In fact, consider the cotangent sequence
    \[
    0\longrightarrow E^\vee_{|Y} \longrightarrow \Omega_{\OGr|Y} \longrightarrow\Omega_Y\longrightarrow 0. 
    \]
    
    Using the Koszul complex twisted by $ E^\vee$, we obtain that the only non-vanishing cohomology of $E^\vee_{|Y}$ is $h^3(Y,E^\vee_{|Y}) = 1$. On the other hand, the Koszul complex twisted by $ \Omega_\OGr$ yields 
    \begin{equation}\label{twistedKoszulOG27}
    0\longrightarrow \Omega_\OGr(-3) \longrightarrow \Omega_\OGr \otimes E(-3)\longrightarrow\Omega_\OGr\otimes E^\vee\longrightarrow \Omega_\OGr \longrightarrow \Omega_{\OGr|Y}\longrightarrow 0.         
    \end{equation}
    Recall that the cotangent bundle of $\OGr(2,7)$ is given by a non-trivial extension  
    \[
    0\longrightarrow \mathcal{O}(-1) \longrightarrow \Omega_\OGr\longrightarrow \mathcal{U}^\vee \otimes \mathrm{Sym}^2\cS (-2)\longrightarrow 0. 
    \]
    Using Bott's theorem, we can compute the cohomology of all bundles appearing in \eqref{twistedKoszulOG27}. However, the cohomology of $ \Omega_\OGr(-3)$ depends on the behavior of the boundary map, which is not explicit. To overcome this, we use the fact that, since $\OGr(2,7)$ is a Fano sevenfold of index $4$, $H^7(\OGr,\Omega_\OGr(-3))= H^0(\OGr,\mathcal{T}_\OGr(-1))$ by Serre duality. By \cite[Theorem 2.2]{FuManivel}, the latter vanishes. Eventually, we obtain that the only non-vanishing cohomologies of $\Omega_{\OGr|Y}$ are $h^1(Y,\Omega_{\OGr|Y}) = 1$ and $h^2(Y,\Omega_{\OGr|Y})=2$. 

    From the cotangent sequence, we obtain the Hodge numbers 
    \[
    \begin{tabular}{c|c c c c}
    $q$ & $0$ & $1$ & $2$ & $3$ \\
    \hline
    $h^q(Y,\Omega_Y)$ & $0$ & $1$ & $3$ & $0$
    \end{tabular} 
    \]

    We can compute the cohomology of the restriction of $E$ to $Y$, and it turns out that it is concentrated in degree zero. This, along with the fact that $H^1(Y,\mathcal{T}_{X|Y})$ is one-dimensional, implies that the map $\varphi$ in \eqref{parametersSequence} is not surjective. Hence, the general element of this family is not of this form.
\end{proof}

\begin{rmk}
    The fact that this family is not locally complete is quite surprising and it seems an interesting question to describe the general element of this family. 
    
    For instance, a similar question is addressed in \cite[Section 3.A]{ManivelDoubleSpinor}. There, it is described a deformation family in which the central fiber is a  Calabi--Yau fivefold obtained as the zero locus of a homogeneous vector bundle in $\OGr(5,10)_+$ while the general fiber is the \textit{double spinor variety}, which is the intersection of two transversal translates of $\OGr(5,10)_+$ in $\mathbb{P}(\Delta_16)$. However, the approach presented there does not seem to work in our case.
\end{rmk}

\subsection{A hyperelliptic curve of genus 3}\label[section]{sec_curve3}
The non-vanishing of the Hodge number $h^{1,2}=3$ in the third family is quite unexpected and it suggests the presence of a curve of genus $3$ related to this fourfold. We show a possible way to construct it.

Fix a general section $f\in H^0(\OGr(2,7),\mathrm{Sym}^2\mathcal{S})\simeq V_{2\omega_3}$, where the latter denotes the irreducible representation of $\mathfrak{so}(7)$ with highest weight $[0,0,2]$. Denote by $Y\subseteq \OGr(2,7)$ the zero locus of $f$. Recall that the Grassmannian $\OGr(2,7)$ corresponds to the marked Dynkin diagram of type $B_3$
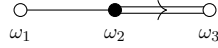
\[
\begin{tikzpicture}[baseline=-3pt]
    \dynkin[label, edge length=1.25cm]{B}{o*o}
  \end{tikzpicture}
\]
Therefore, the same section $f$ defines a fivefold $Z$ of index $4$ in $\OGr(3,7)$ since $H^0(\OGr(3,7),\mathcal{O}(2))\simeq V_{2\omega_3}$ too. Using that $\OGr(3,7) \simeq \OGr(4,8)_+ \simeq (\mathbb{P}^7,\mathcal{O}(2))$, we see that $Z$ is nothing but the intersection of two quadrics $Q_1$ and $Q_2$ in $\mathbb{P}^7$. In particular, its Hodge diamond reads as
$$\xymatrix@1@=3pt@M=0pt{&&&&&1&&&&&\\&&&&0&&0&&&& \\ &&&0&&1&&0&&& \\ &&0&&0&&0&&0&& \\ &0&&0&&1&&0&&0& \\ 0&&0&&3&&3&&0&&0}\\$$
For any $[\lambda,\mu]\in \mathbb{P}^1$, we have a quadric $Q^{[\lambda,\mu]}:= \lambda Q_1 + \mu Q_2$. Consider the incidence defined by the two quadrics
\[
\begin{tikzcd}
             & V := \left\lbrace \left(U_4,[\lambda,\mu]\right)\,|\,Q^{[\lambda,\mu]}_{|U_4}= 0\right\rbrace  \arrow[rd, "p_2"] \arrow[r, hook] \arrow[ld, "p_1"'] & \Gr(4,8)\times \mathbb{P}^1  \\
\Gr(4,8) &                                                                                & \mathbb{P}^1 \supseteq  \Delta_8
\end{tikzcd}
\]
The first projection realizes $V$ as the blow-up of $\Gr(4,8)$ in $Z$. We focus on the second projection. The fiber of $p_2$ away from $\Delta_8$ is $\OGr(4,8)$. The exceptional locus $\Delta_8$ is the set of points $[\lambda,\mu]\in \mathbb{P}^1$ where the quadratic form $Q^{[\lambda,\mu]}$ is degenerate. Since the quadrics $Q_1$ and $Q_2$ are general, $\Delta_8$ consists of eight points. Recall that $\OGr(4,8)$ has two connected components, while the fibers over $\Delta_8$ are connected. Passing to the Stein factorization of $p_2$, we get
\[
\begin{tikzcd}
                   & V \arrow[d, "p_2"'] \arrow[ld, "q"'] \\
V' \arrow[r, "p_2'"'] & \mathbb{P}^1 \supseteq  \Delta_8     
\end{tikzcd}
\]
where the fibers of $q$ are connected and $p_2'$ is a double cover ramified over $\Delta_8$. By Riemann--Hurwitz, $V' $ is a hyperelliptic curve of genus $3$.

\section{Picard rank greater than one}\label[section]{sec_greater1}
When considering a variety obtained as $[X,E]$, it can happen that its Picard rank is strictly bigger than that of the ambient variety $X$, in particular bigger than $1$. Here, we collect those families of Fano fourfolds in generalized Grassmannians for which this phenomenon occurs, as well as the three families \ref{item_oe6}, \ref{item_oe9} and \ref{item_oe11} for which the Picard rank equals that of the ambient Grassmannian $\OGr(3,8)$, which is $2$.

Moreover, for each of these families, we provide birational interpretations that yield interesting birational morphisms between different Fano fourfolds.

\begin{itemize}
    \item The Fano fourfold of type \ref{item_ob3}, i.e. $\left[\OGr(2,11), \cS\oplus \mathrm{Sym}^2\cU^\vee\right]$, has invariants
    $$\xymatrix@1@=3pt@M=0pt{&&&&1&&&&\\&&&0&&0&&& \\ &&0&&2&&0&& \\ &0&&0&&0&&0& \\ 0&&2&&30&&2&&0}$$
    \vspace{0.5pt}
    \[
    (-K)^4 = 80 \quad h^0(-K) = 23 \quad \chi(\mathcal{T}) = -26.
    \]
    By \cite[Proposition 21]{ManivelFrassineti}, this family is the same as $\left[\mathrm{Fl}(1,4,5), \mathcal{O}(2,0) \oplus \mathcal{O}(1,1)\oplus\mathcal{O}(0,2)\right]$. From this model it is easy to see that a fourfold in this family is a complete intersection of two hyperplanes in $Q_1^3\times Q^3_2$. In particular, it is a conic bundle on a three-dimensional quadric.

    \medskip
    \item The Fano fourfold of type \ref{item_ob4}, i.e. $\left[\OGr(2,11), \cS\oplus \mathcal{O}(1)^{\oplus 3}\right]$, has invariants
    $$\xymatrix@1@=3pt@M=0pt{&&&&1&&&&\\&&&0&&0&&& \\ &&0&&2&&0&& \\ &0&&0&&0&&0& \\ 0&&4&&46&&4&&0}$$
    \vspace{0.5pt}
    \[
    (-K)^4 = 70 \quad h^0(-K) = 21 \quad \chi(\mathcal{T}) = -36.
    \]
    Again by \cite[Proposition 21]{ManivelFrassineti}, this family is the same as $\left[\mathrm{Fl}(1,4,5), \mathcal{O}(1,1)^{\oplus 3}\right]$, which is nothing but 
    \[
    \left[\mathbb{P}^4\times \mathbb{P}^4, \mathcal{O}(1,1)^{\oplus 4}\right].
    \]
    Denote by $Y$ a Fano fourfold in this family. We now study one of the two projections from $Y$ to $\mathbb{P}^4$. Denote by $x_1,...,x_5$ the coordinates of the fiber. Then the four global sections of $ \mathcal{O}(1,1)$ give rise to a linear system
    \[\begin{cases}
        x_1f_1^1+x_2f_2^1+\cdots +x_5f_5^1=0\\
        x_1f_1^2+x_2f_2^1+\cdots +x_5f_5^2=0\\
        \vdots\\
        x_1f_1^4+x_2f_2^1+\cdots +x_5f_5^4=0
    \end{cases}
    \]
    where $f_i^j$ are linear forms on the base $\mathbb{P}^4$. The fiber over a point $p\in\mathbb{P}^4$ is a complete intersection of four linear forms in $\mathbb{P}^4$. Therefore, we need to study the kernel of the morphism associated to the matrix $M=((f_i^j)_{ij})$
    \[
    \mathcal{O}_{\mathbb{P}^4}(-1)^{\oplus 5} \xrightarrow{M^t} \mathcal{O}_{\mathbb{P}^4}^{\oplus 4}.
    \]
    Over a general point, this kernel is one-dimensional and thus the fiber is just a point. The first degeneracy locus $S\subseteq\mathbb{P}^4$ has codimension $(4-3)(5-3)=2$, hence $S$ is a surface. Since the codimension of the second degeneracy locus is too big, $S$ is smooth. The fiber of a point of $S$ is isomorphic to $\mathbb{P}^1$ and $Y$ is the blow-up of $\mathbb{P}^4$ in $S$.

    \medskip
    \item The Fano fourfold of type \ref{item_oy1} provides the same family as $$\left[\mathbb{P}^1\times\mathbb{P}^1\times\mathbb{P}^1\times\mathbb{P}^1\times\mathbb{P}^1, \mathcal{O}(1,1,1,1,1) \right]$$ by \cite[Proposition 33]{ManivelFrassineti}. Moreover, by \cite[Theorem 3.1]{kuznetsov}, this is also the family \cite[(d3)]{kuchleList}. In particular, it has invariants
    $$\xymatrix@1@=3pt@M=0pt{&&&&1&&&&\\&&&0&&0&&& \\ &&0&&5&&0&& \\ &0&&0&&0&&0& \\ 0&&1&&26&&1&&0}$$
    \vspace{0.5pt}
    \[
    (-K)^4 = 120 \quad h^0(-K) = 31 \quad \chi(\mathcal{T}) = -16.
    \]
    By \cite[Corollary 3.5]{kuznetsov}, a fourfold in this family is the blow-up of $\mathbb{P}^1\times\mathbb{P}^1\times\mathbb{P}^1\times\mathbb{P}^1$ along the K3 surface
    \[
    S \in \left[\mathbb{P}^1\times\mathbb{P}^1\times\mathbb{P}^1\times\mathbb{P}^1,\mathcal{O}(1,1,1,1)^{\oplus 2}\right].
    \]
    
    \medskip
    \item The Fano fourfold of type \ref{item_ow11} provides the same family as $\left[\Gr(2,7),(\mathrm{Sym}^2 \cU^\vee)^{\oplus 2}\right]$. This is the Fano variety of lines contained in the intersection of two quadrics in $\mathbb{P}^6$ and it has been studied extensively in the literature, see e.g. \cite{reid1972complete}, \cite{BorceakPlanes}, \cite{CasagrandeAraujo}. 
    This is the family \cite[(b9)]{kuchleList}. In particular, the invariants are the following
    $$\xymatrix@1@=3pt@M=0pt{&&&&1&&&&\\&&&0&&0&&& \\ &&0&&8&&0&& \\ &0&&0&&0&&0& \\ 0&&0&&30&&0&&0}$$
    \vspace{0.5pt}
    \[
    (-K)^4 = 80 \quad h^0(-K) = 21 \quad \chi(\mathcal{T}) = -4.
    \]

    \medskip
    \item The Fano fourfolds of type \ref{item_ob7} and \ref{item_ow6} provide, by \cref{fl134} and by triality, the same family as $\left[\mathrm{Fl}(1,3,4),\mathcal{O}(2,2)\right]$ or, equivalently, $$\left[\mathbb{P}^3\times \mathbb{P}^3,\mathcal{O}(1,1)\oplus \mathcal{O}(2,2)\right] \quad \text{or}\quad \left[\mathrm{Gr}(2,6),\mathrm{Sym}^2\mathcal{U}^\vee \oplus \mathcal{O}(2)\right].$$ 
    See \cite[Proposition 2.1]{kuznetsov}. This is the family \cite[(b4)]{kuchleList}. In particular, the invariants are
    $$\xymatrix@1@=3pt@M=0pt{&&&&1&&&&\\&&&0&&0&&& \\ &&0&&2&&0&& \\ &0&&0&&0&&0& \\ 0&&16&&114&&16&&0}$$
    \vspace{0.5pt}
    \[
    (-K)^4 = 40 \quad h^0(-K) = 15 \quad \chi(\mathcal{T}) = -68.
    \]
    By \cite[Corollary 2.2]{kuznetsov}, a fourfold in this family is a conic bundle over $\mathbb{P}^3$.
    
    \end{itemize}

    \bigskip
    We now discuss the three models in $\OGr(3,8)$. Recall that it has Picard rank $2$, generated by the pull-back of line bundles along the diagram
    \[
    \begin{tikzcd}[column sep=small, row sep=large]
    & \OGr(3,8) \arrow[dl, "p_+"'] \arrow[dr, "p_-"] & \\
    \OGr(4,8)_+ & & \OGr(4,8)_-
    \end{tikzcd}
    \]
    In particular, the two spinor bundles are the line bundles
    \[
    \cS_+ = p_+^*\mathcal{O}(1), \quad \cS_- = p_-^*\mathcal{O}(1),
    \]
    while the restriction of $\mathcal{O}_{\Gr(3,8)}(1)$ is equal to $S_+\otimes S_-$ in this case.

    Both projections are $\mathbb{P}^3$-fibrations over the quadric $Q^6$ (triality). 
    
\begin{enumerate}
    \item The Fano fourfold of type \ref{item_oe6}, i.e. $\left[\OGr(3,8), \mathcal{O}(1)\oplus \mathcal{S}_-^{\oplus 2}\oplus \mathcal{S}_+^{\oplus 2}\right]$, has invariants
    $$\xymatrix@1@=3pt@M=0pt{&&&&1&&&&\\&&&0&&0&&& \\ &&0&&2&&0&& \\ &0&&0&&0&&0& \\ 0&&1&&24&&1&&0}$$
    \vspace{0.5pt}
    \[
    (-K)^4 = 100 \quad h^0(-K) = 27 \quad \chi(\mathcal{T}) = -23.
    \]
    Let $Y$ be a fourfold in this family. We claim that $Y$ is the blow-up of a four-dimensional quadric with center the blow-up of a K3 surface in two points. 
    
    Note that $Y$ is a subvariety of the Fano $8$-fold of K3-type studied in \cite[Section 3.7]{FatMonK3Type}, which is a stratified projective bundle over a six-dimensional quadric whose discriminant locus is a K3 surface of genus $7$.
    
    In our case, the image of $p_{+|Y}$ is a four-dimensional quadric $B:=Q^4$. Indeed, by triality, the two sections of $\mathcal{S}_+$ become hyperplane sections of $Q^6$. The remaining three sections give rise to a morphism 
    \[
    \varphi:\mathcal{O}_B \oplus \mathcal{O}_B(1)^{\oplus 2}\longrightarrow (p_+)_* \mathcal{O}(1) = \cU(2).
    \]

    Let $\mathcal{E} := \cU^\vee_{|B}$ and $\mathcal{F}:= \mathcal{O}_B(1)^{\oplus 2}\oplus \mathcal{O}_B(2)$. Then the transpose of $\varphi$ twisted by $\mathcal{O}_B(2)$ yields
    \[
    \varphi^t:\mathcal{E}\longrightarrow \mathcal{F}.
    \]
    We denote by $S\subseteq B$ the first degeneracy locus of $\varphi^t$, which is a smooth surface. By \cite[Lemma 2.1]{kuzc5}, $Y$ is the blow-up of $B$ along $S$. In particular, since we know the Betti numbers of $Y$ and $B$, we obtain that $S$ is a connected surface with irregularity $0$ and $b_2(S) = 24$.

    We claim that this surface $S$ is the blow-up in two points of a K3 surface of genus $7$, which is responsible for the K3 structure in cohomology.

    The Eagon--Northcott complex (see e.g. \cite[Proposition 3.6]{BernadaraFatManTant}) associated to $\varphi^t$ reads as
    \[
    0 \longrightarrow \mathcal{O}_B(-1)^{\oplus 2}\oplus \mathcal{O}_B(-2) \longrightarrow \cU \longrightarrow \mathcal{O}_B(2)\longrightarrow \mathcal{O}_S(2)\longrightarrow 0.
    \]
    The twisted versions by $\mathcal{O}_B(-2)$ and $\mathcal{O}_B(-1)$ allow to compute 
    \[
    \chi(\mathcal{O}_S) = 2, \quad \chi(\mathcal{O}_S(1)) = 6 \quad\text{and}\quad \chi(\mathcal{O}_S(2)) = 20.
    \]
    In particular, the Hilbert polynomial of $S$ is $P_S(m)=10\frac{m^2}{2} -m +2$ and we obtain the intersection numbers
    \[
    H^2=10, \quad H\cdot K_S = 2,
    \]
    where $H:=\mathsf{c}_1(\mathcal{O}_S(1)).$
    Note that we can already deduce the Hodge diamond of $S$ whose upper part reads as
    $$\xymatrix@1@=3pt@M=0pt{&&&&1&&&&\\&&&0&&0&&& \\ &&1&&22&&1&&&}$$

    \medskip
    Moreover, by Noether's formula, we get $K_S^2= -2$. Thus, $K_S$ is not nef and $S$ is not minimal. Let $\sigma:S\longrightarrow T$ be the minimal model of $S$ and assume that it is the contraction of $m$ curves. In particular, we have the relation 
    \[
    K_S = \sigma^*K_T +\sum_{i=1}^m E_i.
    \]
    We want to show that the Kodaira dimension of $T$ is zero. Clearly, $\mathrm{kod}(T)\geq 0$. If $\mathrm{kod}(T)\geq 1$, then $K_T^2\geq 0$ and thus $m\geq 2$. This leads to the following contradiction:
    \[
    2 = H\cdot K_S = H\cdot\sigma^*K_T + \sum_{i=1}^mH\cdot E_i \geq 1+2=3.
    \]
    By the Enriques--Kodaira classification of minimal surface, $T$ is a K3 surface and $\sigma$ is the blow-up of $T$ in two points, i.e. $m=2$. In particular, $K_S = E_1+E_2$ Finally, since $H\cdot K_S = 2$, there is a polarization $H'$ on $T$ such that
    \[
    \sigma^*H'=H+E_1+E_2= H+K_S.
    \]
    It follows that $H'^2=10+4-2=12$, so that $(T,H')$ is a polarized K3 surface of genus $7$.
    
    \medskip
    \item The Fano fourfold of type \ref{item_oe9}, i.e. $\left[\OGr(3,8), \mathcal{S}_-^{\otimes 2}\oplus \mathcal{S}_-\oplus \mathcal{S}_+^{\oplus 3}\right]$, has invariants
    $$\xymatrix@1@=3pt@M=0pt{&&&&1&&&&\\&&&0&&0&&& \\ &&0&&2&&0&& \\ &0&&0&&0&&0& \\ 0&&0&&18&&0&&0}$$
    \vspace{0.5pt}
    \[
    (-K)^4 = 100 \quad h^0(-K) = 27 \quad \chi(\mathcal{T}) = -20.
    \]
    Let $Y$ be a fourfold in this family. The first projection of $\OGr(3,8)$ restricted to $Y$ yields
    \[
    Y \xrightarrow{p_{+|Y}} Q^3.
    \]
    The fiber over any point is a conic, and thus $Y$ is a conic bundle over $Q^3$. 

    \medskip
    \item The Fano fourfold of type \ref{item_oe11}, i.e. $\left[\OGr(3,8), \bigwedge^2\cU^\vee\oplus  \mathcal{S}_-\oplus \mathcal{S}_+\right]$, has invariants
    $$\xymatrix@1@=3pt@M=0pt{&&&&1&&&&\\&&&0&&0&&& \\ &&0&&2&&0&& \\ &0&&0&&0&&0& \\ 0&&0&&12&&0&&0}$$
    \vspace{0.5pt}
    \[
    (-K)^4 = 132 \quad h^0(-K) = 33 \quad \chi(\mathcal{T}) = -12.
    \]

    This family of Fano fourfolds is new to the literature. We now show that $Y$ is the blow-up of the intersection of two quadrics along a Del Pezzo surface of degree $6$. 

    Let $(\omega, s,t)$ be the section defining $Y$. The push-forward of the section $t$ of $\cS_+$ cut a hyperplane $B\subset \OGr(4,8)_+\simeq Q^6$, hence $B$ is a five-dimensional quadric. The push-forward of $\omega$ defines a skew-symmetric map
    \[
    \omega: \cU \longrightarrow \cU^\vee
    \]
    on $\OGr(4,8)_+$. The zero locus of the Pfaffian $\mathrm{Pf}(\omega)$ is a quadric section $B' \subseteq \OGr(4,8)_+$, which is singular in a bunch of isolated points. Hence, we obtain a smooth Fano fourfold $$X:=B\cap B'\subset Q^6$$ which is isomorphic to the intersection of two quadrics in $\mathbb{P}^6$. In particular, it has Hodge numbers:
    $$\xymatrix@1@=3pt@M=0pt{&&&&1&&&&\\&&&0&&0&&& \\ &&0&&1&&0&& \\ &0&&0&&0&&0& \\ 0&&0&&8&&0&&0}$$

    \medskip
    When restricted to $X$, the kernel of $\omega$ is a vector bundle of rank $2$, which we denote by $K$. From
    \[
    0\longrightarrow K \longrightarrow \cU_{|X} \xrightarrow{\omega} \cU^\vee_{|X} \longrightarrow K^\vee \longrightarrow 0,
    \]
    we deduce that $\mathsf{c}_1(K) = \mathsf{c}_1(\cU_{|X})$.
    
    Let $W$ be the zero locus of $(\omega,t)$ in $\OGr(3,8)$ and denote by $\pi$ the restriction $p_{+|W}:W\longrightarrow X$. The fiber over a point $x\in X$ is 
    \[
    \left\lbrace P \in \OGr(3,8) \,|\, K_x \subset P \subset \cU_x\right\rbrace.
    \]
    Therefore, $\pi$ is the $\mathbb{P}^1$-bundle associated to the rank $2$ vector bundle $N := (\cU_{|X}/K)(1)$. 

    In order to take into account the section $s$, we note that its push-forward is a global section of $\cU(1)$ and therefore it defines a map
    \[
    \psi:N^\vee \longrightarrow \mathcal{O}_X.
    \]
    By \cite[Lemma 2.1]{kuzc5}, $Y$ is the blow-up of $X$ along the first degeneracy locus of $\psi$, which is nothing but the zero locus of a global section of $N$. We denote it by $S$. 

    Since $\mathrm{rk}(N)=2$, $S$ is a (smooth) surface and by adjunction
    \[
    K_S = (K_X + \mathsf{c}_1(N))_{|S} = \mathsf{c}_1(\mathcal{O}_S(-3)\otimes\mathcal{O}_S(2)) = \mathsf{c}_1(\mathcal{O}_S(-1)).
    \]
    Therefore, $S$ is a Del Pezzo surface. Since $Y\simeq \mathrm{Bl}_S X$, we deduce that $b_2(S) = 12-8=4$. Hence, $S$ is a Del Pezzo surface of degree $6$.
\end{enumerate}

\bigskip
\begin{rmk}
        In subsequent works, we will present other new families of Fano fourfolds with Picard rank greater than $1$ arising as $\left[G/P,E\right]$ in complete generality. This means that we are allowed to take as ambient variety products of generalized Grassmannians and generalized flag manifolds and to consider zero loci of homogeneous vector bundles over them.

        These works, carried out with other collaborators, are part of a unified project called \textit{FanoBase}. Among the goals are a complete list of families of Fano fourfolds realized in this way, as well as a complete picture of the web of biregular and birational morphisms between such families, such as the ones presented in this section. 
\end{rmk}

\newpage

\renewcommand{\arraystretch}{1.35}
\begin{longtable}{@{}>{\centering\arraybackslash}m{1.6cm} !{\color{black!15}\vrule width 0.4pt} m{\dimexpr\linewidth-1.6cm-2\tabcolsep-0.4pt\relax}@{}}
\caption{Fano fourfold models (cf. \cite[Tables 9-11]{BenedettiTrivial})}\label[table]{TabVlad}\\
\toprule
\endfirsthead
\toprule
\endhead
\bottomrule
\endfoot
\bottomrule
\endlastfoot

\multicolumn{2}{@{}l}{\textbf{Symplectic Grassmannians}} \\
\cmidrule(lr){1-2}
\itemtag{sb2} $\left[\SGr(2,6), (\cU^\perp/\cU)(1) \oplus \mathcal{O}(2)\right]$ \\
\itemtag{sb4} $\left[\SGr(2,8), \cU^\vee \oplus \mathcal{E}_{(0,0,0,1)} \right]$ \\
\itemtag{sb5} $\left[\SGr(2,8), (\cU^\perp/\cU)(1)\oplus \cU^\vee\oplus \mathcal{O}(1)\right]$ \\
\itemtag{sc1} $\left[\SGr(3,8), (\cU^\perp/\cU)(1)\oplus \bigwedge^2\cU^\vee \oplus \cU^\vee \right]$ \\
\addlinespace[8pt]

\multicolumn{2}{@{}l}{\textbf{Odd orthogonal Grassmannians}} \\
\cmidrule(lr){1-2}
\itemtag{ob1} $\left[\OGr(2,9), \cS \oplus \mathrm{Sym}^2\cU^\vee\right]$ \\
\itemtag{ob2} $\left[\OGr(2,9), \cS \oplus \mathcal{O}(1)^{\oplus 3}\right]$ \\
\itemtag{ob3} $\left[\OGr(2,11), \cS\oplus \mathrm{Sym}^2\cU^\vee\right]$ \\
\itemtag{ob4} $\left[\OGr(2,11), \cS\oplus \mathcal{O}(1)^{\oplus 3}\right]$ \\
\itemtag{ox3} $\left[\OGr(2,7), \cS\oplus  \mathcal{O}(2)\right]$ \\
\itemtag{ox5} $\left[\OGr(2,7), \mathrm{Sym}^2\cS\right]$ \\
\itemtag{ox6} $\left[\OGr(3,9), \cS^{\oplus 4}\right]$ \\
\itemtag{oy1} $\left[\OGr(5,11), \bigwedge^2\cU^\vee\oplus \mathcal{O}(1)\right]$ \\
\itemtag{oy2} $\left[\OGr(3,7), \mathcal{O}(3)\oplus \mathcal{O}(2)\right]$ \\
\itemtag{oy3} $\left[\OGr(4,9), \mathcal{O}(1)^{\oplus 5} \oplus \mathcal{O}(2)\right]$ \\
\addlinespace[8pt]

\multicolumn{2}{@{}l}{\textbf{Even orthogonal Grassmannians}} \\
\cmidrule(lr){1-2}
\itemtag{ob5} $\left[\OGr(2,10), \cS_+^{\oplus 2}\oplus \mathcal{O}(2)\right]$ \\
\itemtag{ob7} $\left[\OGr(2,10), \cS_+\oplus \cS_-\oplus \mathcal{O}(2)\right]$ \\
\itemtag{ob9} $\left[\OGr(2,14), \cS_+\oplus \mathcal{O}(2)\right]$ \\
\itemtag{ow6} $\left[\OGr(2,8), \cS_+^{\oplus 2}\oplus \mathcal{O}(2)\right]$ \\
\itemtag{ow7} $\left[\OGr(2,8), \cS_+\oplus \cS_-\oplus \mathcal{O}(2)\right]$ \\
\itemtag{ow10} $\left[\OGr(2,8), \cS_+\oplus \mathrm{Sym}^2\cU^\vee\right]$ \\
\itemtag{ow11} $\left[\OGr(2,8), \cS_+\oplus \mathrm{Sym^2\cS_+}\right]$ \\
\itemtag{ow12} $\left[\OGr(2,8), \cS_+\oplus \mathrm{Sym}^2\cS_-\right]$ \\
\itemtag{ow14} $\left[\OGr(2,8), \cS_+\oplus \mathcal{O}(1)^{\oplus 3}\right]$ \\
\itemtag{oz1} $\left[\OGr(4,8)_+, \mathcal{O}(1)\oplus \mathcal{O}(4)\right]$ \\
\itemtag{oz2} $\left[\OGr(5,10)_+,\mathcal{O}(4)\oplus \cU(1)\right]$ \\
\itemtag{oz4} $\left[\OGr(5,10)_+, \mathcal{O}(1)^{\oplus 5}\oplus\mathcal{O}(2)\right]$ \\
\itemtag{oz5} $\left[\OGr(6,12)_+,\mathcal{O}(1)^{\oplus 5}\oplus \cU(1)\right]$ \\
\itemtag{oe6} $\left[\OGr(3,8), \mathcal{O}(1)\oplus \mathcal{S}_-^{\oplus 2}\oplus \mathcal{S}_+^{\oplus 2}\right]$ \\
\itemtag{oe9} $\left[\OGr(3,8), \mathcal{S}_-^{\otimes 2}\oplus \mathcal{S}_-\oplus \mathcal{S}_+^{\oplus 3}\right]$ \\
\itemtag{oe11} $\left[\OGr(3,8), \bigwedge^2\cU^\vee\oplus  \mathcal{S}_-\oplus \mathcal{S}_+\right]$ \\
\addlinespace[8pt]

\multicolumn{2}{@{}l}{\textbf{Exceptional Grassmannians}} \\
\cmidrule(lr){1-2}
\itemtag{g1} $\left[G_2/P_1,\mathcal{O}(4)\right]$ \\
\itemtag{g2} $\left[G_2/P_2,\mathcal{O}(2)\right]$ \\

\end{longtable}
\renewcommand{\arraystretch}{1}

\vspace*{1.5cm}

\subsection*{Acknowledgment}
This work has been carried out during the author's stay at the \textit{Institut de Math\'ematiques de Marseille} under the supervision of L. Manivel, to whom the author is grateful for guidance and encouragement. The author also wishes to thank E. Fatighenti for many discussions and suggestions during the writing of this paper.

This work is partially supported by INdAM - GNSAGA Project, Varietà di Fano e hyperk\"ahler: costruzioni, classificazioni e collegamenti (CUP E53C24001950001).

\vspace*{1.5cm}

\printbibliography

\end{document}